\documentclass[a4paper,12pt]{amsart}
\usepackage[utf8]{inputenc}
\usepackage[margin=2.5cm]{geometry}
\usepackage{amsmath}
\usepackage{amsfonts}
\usepackage{amssymb}
\usepackage{amsthm}
\usepackage{tikz-cd} 
\usepackage{mathtools}
\usepackage{hyperref} 

\theoremstyle{plain}
\newtheorem*{thm*}{Theorem}
\newtheorem{thm}{Theorem}[section]
\newtheorem{prop}[thm]{Proposition}
\newtheorem{lem}[thm]{Lemma}
\newtheorem{cor}[thm]{Corollary}

\theoremstyle{definition}

\newtheorem{dfn}[thm]{Definition}

\newtheorem{rem}[thm]{Remark}
\newtheorem{ntt}[thm]{Notation}

\numberwithin{equation}{thm}

\newcommand{\beq}{\begin{equation}}
\newcommand{\eeq}{\end{equation}}

\newcommand{\ZZ}{\mathbb{Z}}

\newcommand{\NN}{\mathbb{N}}

\newcommand{\kk}{\Bbbk}

\newcommand{\WW}{\mathbb{W}}

\newcommand{\mc}{\mathcal}

\newcommand{\nonzero}{\setminus \{0\}}

\DeclareMathOperator{\GKdim}{GKdim}

\DeclareMathOperator{\End}{End}
\DeclareMathOperator{\Aut}{Aut}
\DeclareMathOperator{\spn}{span}

\DeclareMathOperator{\Der}{Der}

\DeclareMathOperator{\ab}{ab}

\newcommand{\UU}{U_{\kk(t)}}

\newcommand{\del}{\partial}
\newcommand{\diff}[1]{\frac{d}{d#1}}
\newcommand{\pdiff}[1]{\frac{\del}{\del#1}}

\newcommand\restr[2]{{
  \left.\kern-\nulldelimiterspace 
  #1 
  \vphantom{\big|}
  \right|_{#2}
  }}

\title{Maximal dimensional subalgebras of general Cartan type Lie algebras}

\author{Jason Bell}
\address[Bell]{Department of Pure Mathematics, University of Waterloo, Waterloo, ON, Canada N2L 3G1}
\email{jpbell@uwaterloo.ca}

\author{Lucas Buzaglo}
\address[Buzaglo]{Department of Mathematics, UC San Diego, La Jolla, CA 92093-0112, USA}
\email{lbuzaglo@ucsd.edu}


\thanks{The first author was supported by NSERC grant RGPIN-2022-02951.}

\begin{document}

\begin{abstract}
    Let $\kk$ be a field of characteristic zero and let $\WW_n = \Der(\kk[x_1,\cdots,x_n])$ be the $n^{\text{th}}$ general Cartan type Lie algebra. In this paper, we study Lie subalgebras $L$ of $\WW_n$ of maximal Gelfand--Kirillov (GK) dimension, that is, with $\GKdim(L) = n$.

    For $n = 1$, we completely classify such $L$, proving a conjecture of the second author. As a corollary, we obtain a new proof that $\WW_1$ satisfies the Dixmier conjecture, in other words, $\End(\WW_1) \nonzero = \Aut(\WW_1)$, a result first shown by Du.

    For arbitrary $n$, we show that if $L$ is a GK-dimension $n$ subalgebra of $\WW_n$, then $U(L)$ is not (left or right) noetherian.
\end{abstract}

\maketitle

\section*{Introduction}

Throughout, $\kk$ denotes a field of characteristic zero. All vector spaces are $\kk$-vector spaces. For brevity, we say that a $\kk$-algebra is \emph{noetherian} if it is left and right noetherian.

The \emph{Lie algebras of general Cartan type} (also known as \emph{Cartan type $W$ Lie algebras}), are denoted $\WW_n = \Der(\kk[x_1,\cdots,x_n])$ for $n \geq 1$. For simplicity of notation, we write $\WW_1 = \Der(\kk[t]) = \kk[t]\del$, where $\del = \diff{t}$. Subalgebras of these Lie algebras form an important class of infinite-dimensional Lie algebras, which includes, for example, all the other Cartan type Lie algebras. In this paper, we study Lie subalgebras of $\WW_n$ of maximal Gelfand--Kirillov (GK) dimension, in other words, subalgebras of GK-dimension $n$.

We begin by studying subalgebras of $\WW_1$ of maximal GK-dimension (equivalently, infinite-dimensional subalgebras). Letting $f \in \kk[t] \setminus \kk$, some examples of such subalgebras include:
\begin{enumerate}
    \item $f\WW_1$, known as \emph{submodule-subalgebras}, because they are also $\kk[t]$-submodules of $\WW_1$ under the natural action of $\kk[t]$. These have finite codimension in $\WW_1$.\label{item:submodule-subalgebra}
    \item $L(f) = \WW_1 \cap \Der(\kk[f])$, where this intersection is taken in $\Der(\kk(t)) = \kk(t)\del$ upon identifying $\Der(\kk[f]) = \frac{1}{f'}\kk[f]\del \subseteq \kk(t)\del$. These are not $\kk[t]$-submodules of $\WW_1$, but they are still $\kk[f]$-submodules. If $\deg(f) \geq 2$, then $L(f)$ has infinite codimension in $\WW_1$.\label{item:L(f)}
    \item $\kk[f]$-submodules of $L(f)$. Note that this generalises \eqref{item:submodule-subalgebra} and \eqref{item:L(f)}.\label{item:L(f,g)}
\end{enumerate}
We show that, up to a finite-dimensional vector space, this is an exhaustive list of infinite-dimensional subalgebras of $\WW_1$.

\begin{thm}[Theorem \ref{thm:main}]\label{thm:intro main}
    Let $L$ be an infinite-dimensional subalgebra of $\WW_1$. Then there exist $f \in \kk[t] \setminus \kk$ and a $\kk[f]$-submodule $L'$ of $L(f)$ such that
    $$L' \subseteq L \subseteq L(f).$$
    In particular, $L$ has finite codimension in $L(f)$.
\end{thm}

This establishes \cite[Conjecture 5.1]{Buzaglo}. We remark that the case where $L$ has finite codimension in $\WW_1$ is due to Petukhov and Sierra \cite[Proposition 3.2.7]{PetukhovSierra}.

The inclusion $L \subseteq L(f)$ is \cite[Theorem 5.3]{Buzaglo}, where $f$ is chosen in a ``canonical" way: it is the generator of the \emph{field of ratios} $F(L) = \kk(f)$ (see Definition \ref{dfn:set/field of ratios}). To prove Theorem \ref{thm:intro main}, we show that $L$ has finite codimension in $L(f)$, which is enough to construct $L' \subseteq L$. That $L$ has finite codimension in $L(f)$ is shown by exploiting Hensel's lemma to write an element $g\del \in L$ as $g\del = s^d\del_s$, where $s = t + \text{lower degree terms} \in \kk((t^{-1}))$ and $\del_s = \frac{1}{s'}\del$. We then prove that every element of $L$ can be written as a (possibly infinite) linear combination of terms of the form $s^n\del_s$, from which we deduce that $L$ has finite codimension in $L(f)$.

We then present some applications of Theorem \ref{thm:intro main}. First, we consider the noetherianity of universal enveloping algebras of subalgebras of $\WW_1$. The question of whether it is possible for an infinite-dimensional Lie algebra to have a noetherian universal enveloping algebra is long-standing, having first appeared fifty years ago in Amayo and Stewart's book on infinite-dimensional Lie algebras \cite[Question 27]{AmayoStewart}.

While it is widely believed that enveloping algebras of infinite-dimensional Lie algebras are never noetherian \cite[Conjecture 0.1]{SierraWalton}, there are very few examples whose non-noetherianity is known. One of the most significant milestones since the question was posed by Amayo and Stewart was Sierra and Walton's proof that $U(\WW_1)$ is not noetherian \cite[Theorem 0.5]{SierraWalton}.

As a consequence of the classification of subalgebras from Theorem \ref{thm:intro main}, it follows that an infinite-dimensional subalgebra of $\WW_1$ is isomorphic to a subalgebra of $\WW_1$ of finite codimension. Enveloping algebras of finite codimension subalgebras of $\WW_1$ are known to be non-noetherian \cite{Buzaglo}. This immediately implies that enveloping algebras of infinite-dimensional subalgebras of $\WW_1$ are not noetherian, confirming \cite[Conjecture 4.1]{Buzaglo}.

\begin{cor}[Corollary \ref{cor:noetherian}]\label{cor:intro noetherian}
    Let $L$ be an infinite-dimensional subalgebra of $\WW_1$. Then $U(L)$ is not noetherian.
\end{cor}

Another consequence of Theorem \ref{thm:intro main} is that $\WW_1$ satisfies the Dixmier conjecture. While this had previously been established in \cite{Du}, we provide a more conceptual proof by showing that $\WW_1$ does not have any proper infinite-dimensional simple subalgebras.

\begin{cor}[Corollary \ref{cor:Dixmier}]
    The Lie algebra $\WW_1$ satisfies the Dixmier conjecture. In other words, $\End(\WW_1) \nonzero = \Aut(\WW_1)$.
\end{cor}

This (re)proves the one-variable case of \cite[Conjecture 1]{Zhao}, which says that the Dixmier conjecture holds for all $\WW_n$. Zhao's conjecture, which is open for two or more variables, is exceptionally difficult and has deep implications; Zhao himself proved that it implies the Jacobian conjecture \cite[Theorem 4.1]{Zhao}.

The last part of the paper is devoted to extending the non-noetherianity result of Corollary \ref{cor:intro noetherian} to maximal dimensional subalgebras of $\WW_n$ for arbitrary $n$.

\begin{thm}[Theorem \ref{thm:subalgebras Wn}]\label{thm:intro subalgebras Wn}
    Let $n \geq 1$ and let $L$ be a subalgebra of $\WW_n$ of GK-dimension $n$. Then $U(L)$ is not noetherian.
\end{thm}

The proof of Theorem \ref{thm:intro subalgebras Wn} is achieved by constructing subalgebras $L_{ij}$ of $L$ such that $L_{ij} \supseteq L_{i,j+1}$ for all $i,j$, and studying the dimensions of successive quotients
$$d_{ij}(L) \coloneqq \dim_\kk(L_{ij}/L_{i,j+1}).$$
One of the key steps is deducing that if $\GKdim(L) = n$ and $d_{0,0}(L) = \infty$, then $U(L)$ is not noetherian (see Corollary \ref{cor:d00}). We then show that (possibly after changing variables), if $\GKdim(L) = n$ then we must have $d_{0,0}(L) = \infty$, allowing us to deduce that $U(L)$ is not noetherian.

\vspace{2mm}

\noindent \textbf{Acknowledgements:} This research was done as part of the second author's PhD at the University of Edinburgh, under the supervision of Prof. Susan J. Sierra. Much of this work was carried out during the second author's visit to the University of Waterloo; he would like to thank the institution for their hospitality and support. This collaboration was made possible by the generous funding provided through the University of Edinburgh's Laura Wisewell Travel Scholarship, and NSERC grant RGPIN-2022-02951.

\section{Preliminaries}\label{sec:preliminaries}

We begin by recalling the definition of a derivation of an associative algebra, and defining the Lie algebras of interest in this paper.

\begin{dfn}
    Let $A$ be a $\kk$-algebra. A \emph{derivation} of $A$ is a $\kk$-linear map $D \colon A \to A$ such that $D(ab) = D(a)b + aD(b)$ for all $a,b \in A$. We denote the set of all derivations of $A$ by $\Der(A)$.

    For $n \geq 1$, we write
    $$\WW_n = \Der(\kk[x_1,\cdots,x_n]) = \kk[x_1,\cdots,x_n]\del_1 + \cdots + \kk[x_1,\cdots,x_n]\del_n,$$
    where $\del_i = \pdiff{x_i}$ denotes differentiation with respect to $x_i$. The Lie algebra $\WW_1$ is the \emph{one-sided Witt algebra}. To simplify notation, we write $\WW_1 = \kk[t]\del$, where $\del = \diff{t}$.
\end{dfn}

In this section, we classify infinite-dimensional subalgebras of the one-sided Witt algebra. The Lie bracket of $\WW_1$ is given by
$$[f\del,g\del] = (fg' - f'g)\del$$
for $f,g \in \kk[t]$. It has basis $\{e_n = t^{n+1}\del \mid n \geq -1\}$ such that $[e_n, e_m] = (m - n)e_{n+m}$. We will also consider the Lie algebras $\kk(t)\del$ and $\kk((t^{-1}))\del$ with the obvious Lie brackets, and view $\WW_1$ as a subalgebra of both of these.

We now introduce some noteworthy subalgebras of $\WW_1$. There is a natural action of $\kk[t]$ on $\WW_1$, and the following subalgebras are precisely the subalgebras of $\WW_1$ that are also $\kk[t]$-submodules.

\begin{dfn}
    A \emph{submodule-subalgebra} of $\WW_1$ is a Lie subalgebra of $\WW_1$ which is also a $\kk[t]$-submodule. They are of the form
    $$f\WW_1 = \{g\del \in \WW_1 \mid f \text{ divides } g\},$$
    for $f \in \kk[t] \nonzero$.
\end{dfn}

Note that submodule-subalgebras of $\WW_1$ have finite codimension in $\WW_1$: we have $\dim_\kk(\WW_1/f\WW_1) = \deg(f)$. Subalgebras of $\WW_1$ of finite codimension were classified in \cite{PetukhovSierra}: they are ``very close" to being submodule-subalgebras. Although the result in \cite{PetukhovSierra} considers subalgebras of the \emph{Witt algebra} $W = \Der(\kk[t,t^{-1}])$, the analogous result is true for subalgebras of $\WW_1$ with a nearly identical proof.

\begin{prop}[{\cite[Proposition 3.2.7]{PetukhovSierra}}]\label{prop:Alexey}
    Let $L$ be a subalgebra of $\WW_1$ of finite codimension. Then there exist $f \in \kk[t] \nonzero$ and $n \in \NN$ such that
    $$f^n\WW_1 \subseteq L \subseteq f\WW_1.$$
\end{prop}

Therefore, it remains to classify subalgebras of infinite codimension. In \cite{Buzaglo}, a conjectural classification was given, which we prove in the present paper. In order to state the classification, we need to introduce notation for some subalgebras of $\WW_1$.

\begin{ntt}\label{ntt:L(f,g)}
    For $f,g \in \kk[t] \nonzero$ such that $f'g \in \kk[f]$, we write $L(f,g) = \kk[f]g\del$.
    
    Letting $g_f$ be the unique monic polynomial of minimal degree such that $f'g_f \in \kk[f]$, we write $L(f)$ instead of $L(f,g_f)$. We write $h_f \in \kk[t]$ for the polynomial such that $f'g_f = h_f(f)$.
\end{ntt}

\begin{rem}\label{rem:L(f,g) inclusion}
    \begin{enumerate}
        \item By \cite[Proposition 4.13]{Buzaglo}, $L(f,g) \subseteq L(f)$ for all $f,g \in \kk[t]$ such that $f'g \in \kk[f]$.
        \item We could define $L(f,g) = \kk[f]g\del$ for any pair of polynomials $f,g \in \kk[t]$. However, this is a Lie subalgebra of $\WW_1$ if and only if $f'g \in \kk[f]$ \cite[Lemma 4.12]{Buzaglo}.
    \end{enumerate}
\end{rem}

The Lie algebras $L(f,g)$ seem much more mysterious than submodule-subalgebras, so we briefly explain their origin. The condition that $f'g \in \kk[f]$ implies that $g\del$ is a derivation of $\kk[f]$. Indeed, if $f'g = h(f)$, where $h \in \kk[t]$, then
$$g\del(p(f)) = f'gp'(f) = h(f)p'(f) \in \kk[f],$$
where $p \in \kk[t]$. Hence, $g\del = h(f)\del_f$ as an element of $\Der(\kk[f]) = \kk[f]\del_f$, where $\del_f$ denotes differentiation with respect to $f$. Consequently, we introduce the following notation.

\begin{ntt}
    For $f \in \kk[t] \setminus \kk$, we write $\WW[f] = \Der(\kk[f])$. Submodule-subalgebras of $\WW[f]$ are denoted by $g(f)\WW[f] = g(f)\kk[f]\del_f = \{h(f)\del \mid g \text{ divides } h\}$ (these are $\kk[f]$-submodules of $\WW[f]$).
\end{ntt}

Clearly, $\WW[f] \cong \WW_1$ and $h(f)\WW[f] \cong h\WW_1$ as Lie algebras. By the above discussion, $L(f,g)$ can be viewed as the submodule-subalgebra $h(f)\WW[f]$ of $\WW[f]$, where $h \in \kk[t]$ such that $f'g = h(f)$. We therefore immediately see that $L(f,g) \cong h\WW_1$.

The Lie algebra $\WW[f]$ can be naturally viewed as a subalgebra of $\kk(t)\del$ by identifying $\del_f$ with $\frac{1}{f'}\del$, which is the unique extension of $\del_f$ to a derivation of $\kk(t)$. In other words, we make the identification $\WW[f] = \frac{1}{f'}\kk[f]\del \subseteq \kk(t)\del$. Now, $L(f) \subseteq \WW_1 \cap \WW[f]$, viewing this as an intersection in $\kk(t)\del$. It is therefore natural to ask whether $L(f) = \WW_1 \cap \WW[f]$. The following result answers this question positively.

\begin{lem}
    Let $f \in \kk[t] \setminus \kk$. Then $L(f) = \WW_1 \cap \WW[f]$.
\end{lem}
\begin{proof}
    As noted above, $L(f) \subseteq \WW_1 \cap \WW[f]$, so it remains to show that $\WW_1 \cap \WW[f] \subseteq L(f)$.

    Let $u \in \WW_1 \cap \WW[f] \nonzero$. Since $u \in \WW[f] = \frac{1}{f'}\kk[f]\del$, there exists $h \in \kk[t]$ such that $u = \frac{h(f)}{f'}\del$. Similarly, $u$ being an element of $\WW_1$ implies that $u = g\del$ for some $g \in \kk[t]$. Therefore, we have $f'g = h(f) \in \kk[f]$, and thus $u = g\del \in L(f,g) \subseteq L(f)$, where this last inclusion is from Remark \ref{rem:L(f,g) inclusion}.
\end{proof}

We now state the classification of subalgebras of $\WW_1$ first conjectured in \cite{Buzaglo}.

\begin{thm}[{cf. \cite[Conjecture 5.1]{Buzaglo}}]\label{thm:main}
    Let $L$ be an infinite-dimensional subalgebra of $\WW_1$. Then there exist $f,g \in \kk[t]$ such that $f'g \in \kk[f]$ and
    $$L(f,g) \subseteq L \subseteq L(f).$$
    In particular, $L$ has finite codimension in $L(f)$.

    Furthermore, $f$ are $g$ are unique under the following assumptions:
    \begin{itemize}
        \item $f$ is monic and $f(0) = 0$;
        \item $g$ is monic and has minimal degree such that $L(f,g) \subseteq L$.
    \end{itemize}
    In other words, if $\widetilde{f}$ and $\widetilde{g}$ are polynomials such that $L(\widetilde{f},\widetilde{g}) \subseteq L \subseteq L(\widetilde{f})$, then $\widetilde{f} = \alpha f + \beta$ and $\widetilde{g} = p(f)g$, for some $\alpha,\beta \in \kk$ with $\alpha \neq 0$ and $p \in \kk[t]$. In particular, $L(\widetilde{f}) = L(f)$ and $L(\widetilde{f},\widetilde{g}) \subseteq L(f,g)$.
\end{thm}

As noted in \cite{Buzaglo}, in order to prove Theorem \ref{thm:main}, it suffices to prove that if $L$ is an infinite-dimensional subalgebra of $\WW_1$, then there exists a polynomial $f \in \kk[t] \nonzero$ such that $L$ has finite codimension in $L(f)$. For completeness, we give a proof of this fact.

\begin{lem}\label{lem:finite codimension in L(f)}
    Let $f \in \kk[t] \nonzero$ and let $L$ be a subalgebra of $L(f)$ of finite codimension. Then there exists a polynomial $g \in \kk[t] \nonzero$ such that $L(f,g) \subseteq L$.
\end{lem}
\begin{proof}
    Let $h \in \kk[t]$ such that $f'g_f = h(f)$. There is an isomorphism
    \begin{align*}
        \varphi \colon L(f) &\to h\WW_1 \\
        p(f)g_f\del &\mapsto p(t)h\del,
    \end{align*}
    for $p \in \kk[t]$. Now, $\varphi(L)$ is a subalgebra of $h\WW_1$ of finite codimension, so by Proposition \ref{prop:Alexey}, there exists $p \in \kk[t]$ such that $ph\WW_1 \subseteq \varphi(L)$. Let $g = p(f)g_f$. Note that $f'g \in \kk[f]$ and $\varphi^{-1}(ph\WW_1) = L(f,g)$, which immediately implies that $L(f,g) \subseteq L$.
\end{proof}

For the next result, we will need the following notation.

\begin{dfn}\label{dfn:set/field of ratios}
    Let $L$ be a subalgebra of $\WW_1$. The \emph{set of ratios} of $L$ is
    $$R(L) = \left\{\frac{p}{q} \in \kk(t) \mid p\del, q\del \in L, q \neq 0\right\}.$$
    The \emph{field of ratios} of $L$, denoted $F(L)$, is the subfield of $\kk(t)$ generated by $R(L)$.
\end{dfn}

The following result from \cite{Buzaglo} will be crucial in our proof of Theorem \ref{thm:main}.

\begin{prop}[{\cite[Theorem 5.3]{Buzaglo}}]\label{prop:contained in L(f)}
    Let $L$ be an infinite-dimensional subalgebra of $\WW_1$. Then there exists a polynomial $f \in \kk[t] \nonzero$ such that $F(L) = \kk(f)$ and $L \subseteq L(f)$.
\end{prop}

By Lemma \ref{lem:finite codimension in L(f)} and Proposition \ref{prop:contained in L(f)}, it suffices to prove that if $L$ is an infinite-dimensional subalgebra of $\WW_1$ and $F(L) = \kk(f)$, where $f \in \kk[t]$, then $L$ has finite codimension in $L(f)$. In order to prove this, we will analyse the degrees of elements of $L$.

\begin{ntt}
    For $f \in \kk[t]$, we let the \emph{degree} of $f\del \in \WW_1$ be $\deg(f\del) = \deg(f) - 1$. For a subalgebra $L$ of $\WW_1$, we write $d(L) = \gcd\{\deg(u) \mid u \in L\}$.
\end{ntt}

The following lemma shows that an infinite-dimensional subalgebra $L \subseteq \WW_1$ contains elements whose degrees are arbitrarily large multiples of $d(L)$.

\begin{lem}[{\cite[Lemma 4.7]{Buzaglo}}]\label{lem:large degree}
    Let $L$ be an infinite-dimensional subalgebra of $\WW_1$ and let $d = d(L)$. Then there exists $n \in \NN$ such that for all $m \geq n$, there is an element $u_m \in L$ such that $\deg(u_m) = md$.
\end{lem}

Let $L$ be an infinite-dimensional subalgebra of $\WW_1$ and let $f \in \kk[t]$ such that $F(L) = \kk(f)$ and $L \subseteq L(f)$, which exists by Proposition \ref{prop:contained in L(f)}. It is clear that $d(L(f)) = \deg(f)$. By Lemma \ref{lem:large degree}, if $\deg(f) = d(L)$ then $L$ has finite codimension in $L(f)$. Therefore, in order to prove Theorem \ref{thm:main}, it suffices to prove that $\deg(f) = d(L)$. The next section is devoted to proving this fact.

\section{Proof of Theorem \ref{thm:main}}

As noted above, the goal of this section is to prove the following result.

\begin{prop}\label{prop:degree of f}
    Let $L$ be an infinite-dimensional subalgebra of $\WW_1$ and let $f \in \kk[t]$ such that $F(L) = \kk(f)$ and $L \subseteq L(f)$. Then $\deg(f) = d(L)$.
\end{prop}

The next result is an application of Hensel's lemma which shows that any element $u \in \WW_1$ can be written as $u = s^d\del_s$ for some $s \in \kk((t^{-1}))$ of degree 1, where $d = \deg(u) + 1$. Note that, as before, we make the identification $\del_s = \frac{1}{s'}\del$.

This will be used as follows: if $L$ is an infinite-dimensional subalgebra of $\WW_1$, then we can choose $u \in L \nonzero$ and write it as $u = s^d\del_s$, where $s$ and $d$ are as above. It will then follow that all elements of $L$ are spanned by elements of the form $s^{kd(L)+1}\del_s$ for $k \in \NN$. From this, we deduce Proposition \ref{prop:degree of f}.

\begin{lem}\label{lem:monomial}
    If $g \in \kk[t]$ is monic of degree $d \ge 1$, then there exists $s \in \kk((t^{-1}))$ with $s = t + {\rm lower~degree~terms}$ such that $g\del = s^d \del_s = \frac{s^d}{s'}\del$.
\end{lem}
\begin{proof}
    The proof is achieved by Hensel's lemma. We claim that for all $n \ge 0$, there exist $s_n \in \kk[t,t^{-1}]$ such that $s_{i+1} = s_{i} + O(t^{-i})$ and $s_i^d \del_{s_i} = (g + O(t^{d-i-1}))\del$ for $i \ge 0$, or equivalently, $s_i^d/s_i' = g + O(t^{d-i-1})$. We prove this by induction. We let $s_0 = t$. Then $s_0^d \partial_{s_0} = (g + O(t^{d-1}))\del$, so we have the base case.
    
    Now suppose we have constructed $s_0,\cdots,s_n$. Then $s_n^d/s_n' = (g + O(t^{d-n-1}))$, so, noting that $s_n' = 1 + O(t^{-1})$,
    \begin{equation}\label{eq:s_n}
        s_n^d = s_n'(g + \alpha t^{d-n-1}) + O(t^{d-n-2})
    \end{equation}
    for some $\alpha \in \kk$. We let $s_{n+1} = s_n + c t^{-n}$, where $c \in \kk$ is a constant to be determined later. Then
    $$s_{n+1}^d = (s_n + c t^{-n})^d = s_n^d + cd t^{d-n-1} + O(t^{d-n-2}).$$
    By \eqref{eq:s_n}, we have $s_{n+1}^d = s_n'(g + \alpha t^{d-n-1}) + cd t^{d-n-1} + O(t^{d-n-2})$. Since $s_n' = 1 + O(t^{-1})$, it follows that
    $$s_{n+1}^d = s_n'(g + (\alpha + cd)t^{d-n-1}) + O(t^{d-n-2}).$$
    Noting that $s_n' = s_{n+1}' + cn t^{-n-1}$, we get
    \begin{align*}
        s_{n+1}^d &= (s_{n+1}' + cn t^{-n-1})(g + (\alpha + cd)t^{d-n-1}) + O(t^{d-n-2}) \\
        &= s_{n+1}'(g + (\alpha + c(n + d))t^{d-n-1} + O(t^{d-n-2})),
    \end{align*}
    where we used that $g = t^d + O(t^{d-1})$ and $s_{n+1}' = 1 + O(t^{-1})$. Therefore, taking $c = -\frac{\alpha}{n+d}$, we see that $s_{n+1}^d/s_{n+1}' = g + O(t^{d-n-2})$, as required.
\end{proof}

Notice that the $s$ we have constructed is of the form 
$t + {\rm lower~degree~terms}$. In particular, $\kk((t^{-1})) = \kk((s^{-1}))$ and so $\kk((s^{-1}))\del_s = \kk((t^{-1}))\del$.  Moreover, $\WW_1$ is a subalgebra of $\kk((t^{-1}))\del$, so we can work in the Lie algebra $\kk((s^{-1}))\del_s$. We will consider the following subalgebras of $\kk((s^{-1}))\del_s$, which we refer to as \emph{Veronese subalgebras}.

\begin{ntt}
    Let $s \in \kk((t^{-1}))$ such that $s = t + {\rm lower~degree~terms}$ and let $d \geq 1$. We let $V_d(s) = \kk((s^{-d}))s\del_s$. If $s = t$, we simply write $V_d$ instead of $V_d(t)$.
\end{ntt}

The Veronese subalgebra $V_d$ for $d \geq 1$ is simply the subalgebra of $\kk((t^{-1}))\del$ whose elements are of the form $\sum_{k = -\infty}^n \alpha_k e_{kd}$, where $\alpha_k \in \kk$ and $n \in \ZZ$. In other words, when we write elements of $V_d$ as (possibly infinite) linear combinations of the elements $e_n$, the only terms that appear are those whose indices are multiples of $d$.

Similarly, the Lie algebras $V_d(s)$ consist of elements of the form $\sum_{k=-\infty}^n \alpha_k s^{kd+1}\del_s$, where $\alpha_k \in \kk$ and $n \in \ZZ$, which is precisely what we get from $V_d$ under the change of variables $t \mapsto s$.

Given an infinite-dimensional subalgebra $L \subseteq \WW_1$, we would like to show that $L \subseteq V_{d(L)}(s)$ for some suitable choice of $s \in \kk((t^{-1}))$, which would immediately imply that $F(L) \subseteq \kk((s^{-d(L)}))$. Letting $f \in \kk[t]$ such that $F(L) = \kk(f)$, it would then follow that $\deg(f) = d(L)$.

The following result gives useful restrictions for Lie algebras $L$ not contained in $V_{d(L)}$. When we write $e_n + \cdots$, we mean $e_n + {\rm lower~degree~terms}$.

\begin{prop}\label{prop:Veronese}
    Let $L$ be a subalgebra of $\kk((t^{-1}))\del$ and let $d = d(L)$. Suppose $L$ is not contained in $V_d$. Let $k \in \NN$ be minimal such that there is an element
    $$a = e_n + \cdots + \alpha e_{n-k} + \cdots \in L,$$
    where $n = rd$ for some $r \in \NN$, $d$ does not divide $k$, and $\alpha \in \kk \nonzero$. Let $b \in L$ be a monic element such that $m = \deg(b) \neq n$, and let $\beta \in \kk$ be the coefficient of $e_{m-k}$ in $b$ (so $b = e_m + \cdots + \beta e_{m-k} + \cdots$, where $m = r'd$ for some $r' \in \NN$). Then $(m + k)\alpha = (n + k)\beta$. In particular, $\beta \neq 0$.
\end{prop}
\begin{proof}
    \textbf{Step 1.} We begin by proving this when $m = 2n$, i.e.
    \begin{align*}
        a &= e_n + \cdots + \alpha e_{n-k} + \cdots, \\
        b &= e_{2n} + \cdots + \beta e_{2n-k} + \cdots.
    \end{align*}
    By taking Lie brackets of $a$ with $b$, we get two different elements of degree $5n$ in $L$, as follows:
    \begin{align*}
        [a,b] &= ne_{3n} + \cdots + ((n + k)\alpha + (n - k)\beta)e_{3n-k} + \cdots \\
        [a,[a,b]] &= 2n^2 e_{4n} + \cdots + ((4n^2 + 2kn - k^2)\alpha + (2n^2 - 3kn + k^2)\beta)e_{4n-k} + \cdots \\
        [a,[a,[a,b]]] &= 6n^3 e_{5n} + \cdots \\
        &\quad + ((18n^3 + 4kn^2 - 5k^2n + k^3)\alpha + (6n^3 - 11kn^2 + 6k^2n - k^3)\beta)e_{5n-k} + \cdots \\
        [b,[a,b]] &= n^2e_{5n} + \cdots + ((n^2 - k^2)\alpha + (2n^2 - kn + k^2)\beta)e_{5n-k} + \cdots,
    \end{align*}
    where we used the minimality of $k$ to ensure there are no other contributions to the coefficients. Let $c = [a,[a,[a,b]]] - 6n[b,[a,b]] \in L$. We see that $c$ is an element of degree less than $5n$. Since $d$ does not divide $5n - k$, the degree of $c$ cannot be $5n - k$ (recall that $d$ divides the degree of every element of $L$). Therefore, the coefficient of $e_{5n-k}$ in $c$ must be 0, by minimality of $k$. Hence, we have
    $$(18n^3 + 4kn^2 - 5k^2n + k^3)\alpha + (6n^3 - 11kn^2 + 6k^2n - k^3)\beta = 6n((n^2 - k^2)\alpha + (2n^2 - kn + k^2)\beta),$$
    and therefore
    \begin{equation}\label{eq:2n + k}
        (12n^3 + 4kn^2 + k^2n + k^3)\alpha = (6n^3 + 5kn^2 + k^3)\beta.
    \end{equation}
    We have
    \begin{align*}
        12n^3 + 4kn^2 + k^2n + k^3 &= (2n + k)(6n^2 - kn + k^2), \\
        6n^3 + 5kn^2 + k^3 &= (n + k)(6n^2 - kn + k^2).
    \end{align*}
    We claim that $6n^2 - kn + k^2 \neq 0$. Note that $-1 \leq n - k < n$. If $k < n$, then $6n^2 - kn + k^2 > 5n^2 + k^2 > 0$. Since $d$ does not divide $k$, we cannot have $k = n$. Therefore, the only other possibility is $k = n + 1$. In this case, we have $6n^2 - kn + k^2 = 6n^2 + n + 1 > 0$. This proves the claim.
    
    Dividing \eqref{eq:2n + k} by $(6n^2 - kn + k^2)$, we get
    $$(2n + k)\alpha = (n + k)\beta.$$
    \textbf{Step 2.} We now prove the result for a general $m \in \mathbb{N}$. We have
    \begin{align*}
        [a,b] &= (m - n)e_{n+m} + \cdots + ((m - n + k)\alpha + (m - n - k)\beta)e_{n+m-k} + \cdots \\
        [a,[a,b]] &= m(m - n)e_{2n+m} + \cdots \\
        &+ ((2m^2 - 2nm + km - k^2)\alpha + (m^2 - nm + kn - 2km + k^2)\beta)e_{2n+m-k} + \cdots \\
        [b,[a,[a,b]]] &= 2nm(m - n)e_{2(n+m)} + \cdots \\
        &+ (4nm^2 - 4n^2m - 2km^2 + 4knm - 2k^2n - k^2m + k^3)\alpha e_{2(n+m) - k} \\
        &+ (4nm^2 - 4n^2m + 2kn^2 - 4knm + k^2n + 2k^2m - k^3)\beta e_{2(n+m) - k} + \cdots.
    \end{align*}
    Let $a' = \frac{1}{m-n}[a,b] \in L$, and let $b' = \frac{1}{2nm(m - n)}[b,[a,[a,b]]] \in L$. Letting $\alpha'$ be the coefficient of $e_{n+m-k}$ in $a'$ and $\beta'$ be the coefficient of $e_{2(n+m)-k}$ in $b'$, we see that
    $$(2(n + m) + k)\alpha' = (n + m + k)\beta',$$
    by Step 1. Therefore, we have
    \begin{align*}
        &\frac{(2n + 2m + k)(m - n + k)}{m - n}\alpha + \frac{(2n + 2m + k)(m - n - k)}{m - n}\beta \\
        &\hspace{3cm} = \frac{(n + m + k)(4nm^2 - 4n^2m - 2km^2 + 4knm - 2k^2n - k^2m + k^3)}{2nm(m - n)}\alpha \\
        &\hspace{3cm} \quad + \frac{(n + m + k)(4nm^2 - 4n^2m + 2kn^2 - 4knm + k^2n + 2k^2m - k^3)}{2nm(m - n)}\beta,
    \end{align*}
    and hence
    \begin{align}\label{eq:m + k}
        (2km^3 + 2kn^2m + 2k^2n^2 + 3&k^2m^2 + k^2nm + k^3n - k^4)\alpha \nonumber \\
        &= (2kn^3 + 2knm^2 + 3k^2n^2 + 2k^2m^2 + k^2nm + k^3m - k^4)\beta.
    \end{align}
    Note that
    \begin{align*}
        2km^3 + 2kn^2m + 2k^2n^2 + 3k^2m^2 + k^2nm &+ k^3n - k^4 \\
        &= k(2n^2 + 2m^2 + kn + km - k^2)(m + k), \\
        2kn^3 + 2knm^2 + 3k^2n^2 + 2k^2m^2 + k^2nm &+ k^3m - k^4 \\
        &= k(2n^2 + 2m^2 + kn + km - k^2)(n + k).
    \end{align*}
    Since $n \neq m$, we must have that $k < n$ or $k < m$, and therefore $k < n + m$. Hence, we have
    $$2n^2 + 2m^2 + kn + km - k^2 = 2n^2 + 2m^2 + (n + m - k)k > 0.$$ Dividing \eqref{eq:m + k} by $k(2n^2 + 2m^2 + kn + km - k^2)$, we conclude that
    $$(m + k)\alpha = (n + k)\beta,$$
    as required.
\end{proof}
  
As an immediate consequence of the proposition above, we get a surprising criterion for a subalgebra $L \subseteq \WW_1$ to be contained in $V_{d(L)}(s)$.
  
\begin{cor}\label{cor:criterion for containment}
    Let $L$ be an infinite-dimensional subalgebra of $\kk((t^{-1}))\del$, let $s \in \kk((t^{-1}))$ such that $s = t + {\rm lower~degree~terms}$, and write $d = d(L)$. Then $L \cap V_d(s) \neq 0$ if and only if $L \subseteq V_d(s)$.
\end{cor}
\begin{proof}
    Suppose $L$ is not contained in $V_d(s)$. Let $k \in \NN$ be minimal such that there is an element
    $$a = s^{n+1}\del_s + \cdots + \alpha s^{n-k+1}\del_s + \cdots \in L,$$
    where $n = rd$ for some $r \in \NN$, $d$ does not divide $k$, and $\alpha \in \kk \nonzero$. Upon changing variables $t \mapsto s$, Proposition \ref{prop:Veronese} implies that no nonzero element of $L$ can be contained in $V_d(s)$, which concludes the proof.
\end{proof}

Combining Lemma \ref{lem:monomial} and Corollary \ref{cor:criterion for containment}, we obtain the following result.

\begin{cor}\label{cor:contained in V}
    Let $L$ be an infinite-dimensional subalgebra of $\WW_1$ and let $d = d(L)$. Then there exists $s \in \kk((t^{-1}))$ such that $s = t + {\rm lower~degree~terms}$ and $L \subseteq V_d(s)$.
\end{cor}
\begin{proof}
    Let $u \in L \nonzero$. By Lemma \ref{lem:monomial}, there exists $s \in \kk((t^{-1}))$ such that $s = t + {\rm lower~degree~terms}$ and $u = s^n\del_s$, where $n = \deg(u) + 1$. Now, $d$ divides $\deg(u) = n - 1$, so $n = kd + 1$ for some $k \in \NN$. Therefore, $u = s^{kd+1}\del_s \in V_d(s)$, and thus $L \cap V_d(s) \neq 0$. By Corollary \ref{cor:criterion for containment}, we have $L \subseteq V_d(s)$.
\end{proof}

We are now ready to prove Proposition \ref{prop:degree of f}.

\begin{proof}[Proof of Proposition \ref{prop:degree of f}]
    Let $d = d(L)$ and let $f \in \kk[t]$ such that $F(L) = \kk(f)$. Note that $\deg(f) \leq d$: by Lemma \ref{lem:large degree}, there exist elements $p\del,q\del \in L \nonzero$ such that $\deg(p) = \deg(q) + d$. Now, $\frac{p}{q} \in \kk(f)$, which would be impossible if $\deg(f) > d$. Therefore, it suffices to prove that $\deg(f) \geq d$.
    
    By Corollary \ref{cor:contained in V}, there exists $s = t + \cdots \in \kk((t^{-1}))$ such that $L \subseteq V_d(s) = \kk((s^{-d}))s\del_s$. Therefore, $F(L) \subseteq \kk((s^{-d}))$, so $f \in \kk((s^{-d}))$.

    Note that the valuation on $\kk((t^{-1}))$ given by taking the order of the pole at $\infty$ has the property that every element of $\kk((s^{-d}))$ has valuation a multiple of $d$. On the other hand, this valuation applied to $f$ is just the degree of $f$, so $f$ has degree a multiple of $d$. Hence, $\deg(f) \geq d$, which completes the proof.
\end{proof}

The proof of Theorem \ref{thm:main} now follows.

\begin{proof}[Proof of Theorem \ref{thm:main}]
    Let $f \in \kk[t]$ be such that $F(L) = \kk(f)$, which exists by Proposition \ref{prop:contained in L(f)}. Certainly, we can choose $f$ to be monic and $f(0) = 0$, since $\kk(f) = \kk(\alpha f + \beta)$ for all $\alpha, \beta \in \kk$ with $\alpha \neq 0$. As noted at the end of Section \ref{sec:preliminaries}, to prove that $L$ has finite codimension in $L(f)$, it suffices to prove that $\deg(f) = d(L)$. This is Proposition \ref{prop:degree of f}.
    
    We now prove that $f$ is unique. Suppose $\widetilde{f} \in \kk[t]$ is such that $L$ has finite codimension in $L(\widetilde{f})$. It must be the case that $F(L) \subseteq \kk(\widetilde{f})$. By Lemma \ref{lem:finite codimension in L(f)}, there exists $\widetilde{g} \in \kk[t] \nonzero$ such that $L(\widetilde{f},\widetilde{g}) \subseteq L$, and thus $\kk(\widetilde{f}) \subseteq F(L)$. We conclude that
    $$\kk(f) = F(L) = \kk(\widetilde{f}).$$
    This equality implies that the map
    \begin{align*}
        \varphi \colon \kk(f) &\to \kk(\widetilde{f}) = \kk(f) \\
        f &\mapsto \widetilde{f}
    \end{align*}
    is an automorphism. Therefore, $\varphi$ is a M\"obius transformation, so there exist $a,b,c,d \in \kk$ with $ab - cd \neq 0$ such that
    $$\widetilde{f} = \varphi(f) = \frac{af + b}{cf + d}.$$
    But $f$ and $\widetilde{f}$ are non-constant polynomials, so it must be the case that $c = 0$. Letting $\alpha = \frac{a}{d}$ and $\beta = \frac{b}{d}$, we see that $\widetilde{f} = \alpha f + \beta$, which proves uniqueness of $f$.

    Now let $g \in \kk[t] \setminus \{0\}$ be a monic polynomial of minimal degree such that $f'g \in \kk[f]$ and $L(f,g) \subseteq L$ (which exists by Lemma \ref{lem:finite codimension in L(f)}), and suppose $\widetilde{g}$ is another polynomial such that $f'\widetilde{g} \in \kk[f]$ and $L(f,\widetilde{g}) \subseteq L$. Let $q,\widetilde{q} \in \kk[t]$ such that
    $$f'g = q(f), \quad f'\widetilde{g} = \widetilde{q}(f),$$
    and let $r = \gcd(q,\widetilde{q})$. By B\'ezout's lemma, there exist $u,v \in \kk[t]$ such that $uq + v\widetilde{q} = r$. We have
    $$r(f) = (u(f)q(f) + v(f)\widetilde{q}(f)) = (u(f)g + v(f)\widetilde{g})f'.$$
    Letting $h = \frac{r(f)}{f'} = u(f)g + v(f)\widetilde{g}$, we see that $f'h \in \kk[f]$ and
    $$L(f,h) = \kk[f]h\del = \kk[f](u(f)g + v(f)\widetilde{g})\del \subseteq \kk[f]g\del + \kk[f]\widetilde{g}\del = L(f,g) + L(f,\widetilde{g}) \subseteq L.$$
    By minimality of $\deg(g)$, it must be the case that $\deg(h) \geq \deg(g)$, which implies that $\deg(r) \geq \deg(q)$. But $r = \gcd(q,\widetilde{q})$, so it follows that $r = q$. Therefore, $q$ divides $\widetilde{q}$, so $\widetilde{q} = pq$ for some $p \in \kk[t]$. Thus, 
    $$\widetilde{g} = \frac{\widetilde{q}(f)}{f'} = \frac{p(f)q(f)}{f'} = p(f)g$$
    and the uniqueness of $g$ follows.
\end{proof}

\section{Consequences of Theorem \ref{thm:main}}

In this section, we collect some consequences of Theorem \ref{thm:main}. The first result follows from Theorem \ref{thm:main} and the observation that for all $f \in \kk[t] \setminus \kk$, we have $L(f) \cong h_f\WW_1$.

\begin{cor}\label{cor:finite codimension isomorphism}
    Let $L$ be an infinite-dimensional subalgebra of $\WW_1$. Then $L$ is isomorphic to a subalgebra of $\WW_1$ of finite codimension.
\end{cor}
\begin{proof}
    By Theorem \ref{thm:main}, there exists $f \in \kk[t] \setminus \kk$ such that $L$ has finite codimension in $L(f)$. Since $L(f)$ is isomorphic to a subalgebra of $\WW_1$ of finite codimension, so is $L$.
\end{proof}

One of the main questions in \cite{Buzaglo} is whether $U(L)$ is noetherian if $L$ is an infinite-dimensional subalgebra of $\WW_1$. Note that when we say \emph{noetherian}, we mean left and right noetherian.

The ring $U(L)$ was already known to be non-noetherian if $L$ is a subalgebra of $\WW_1$ of finite codimension, but the general case was still open. Thanks to Corollary \ref{cor:finite codimension isomorphism}, the general case follows immediately.

\begin{cor}[{cf. \cite[Conjecture 4.1]{Buzaglo}}]\label{cor:noetherian}
    Let $L$ be an infinite-dimensional subalgebra of $\WW_1$. Then $U(L)$ is not noetherian.
\end{cor}
\begin{proof}
    By \cite[Theorem 0.5]{SierraWalton}, $U(\WW_1)$ is not noetherian. Now, $L$ is isomorphic to a subalgebra of $\WW_1$ of finite codimension by Corollary \ref{cor:finite codimension isomorphism}, so \cite[Proposition 2.1]{Buzaglo} implies that $U(L)$ is not noetherian.
\end{proof}

We finish the section by considering endomorphisms of $\WW_1$. Zhao conjectured that the Lie algebras $\WW_n$ satisfy the Dixmier conjecture for $n \geq 1$ \cite[Conjecture 1]{Zhao}, in other words, that nonzero endomorphisms of $\WW_n$ are automorphisms. The one-variable case was proved by Du \cite{Du}, but the question remains open for two or more variables. This is an extremely difficult question with deep consequences: Zhao showed that his $n$-variable conjecture implies the $n$-dimensional Jacobian conjecture \cite[Theorem 4.1]{Zhao}. In turn, the $2n$-dimensional Jacobian conjecture implies the Dixmier conjecture for the $n^{\text{th}}$ Weyl algebra \cite{Tsuchimoto,Belov-KanelKontsevich}, so Zhao's conjecture for all $n$ also implies the Dixmier conjecture for Weyl algebras.

We give an alternative and more conceptual proof of Du's result, which follows by showing that proper infinite-dimensional subalgebras of $\WW_1$ are not simple.

\begin{prop}\label{prop:simple subalgebra}
    Let $L$ be an infinite-dimensional simple Lie subalgebra of $\WW_1$. Then $L = \WW_1$.
\end{prop}
\begin{proof}
    Note that $L$ cannot be contained in a Lie algebra isomorphic to $f\WW_1$ for any $f \in \kk[t] \setminus \kk$: this is because $f\WW_1$ does not contain any simple subalgebras. An easy way to see this is that nonzero subalgebras of $f\WW_1$ are never perfect, since $[f\WW_1,f\WW_1] = f^2\WW_1$.
    
    Now, Theorem \ref{thm:main} implies that $L$ has finite codimension in $L(f)$ for some $f \in \kk[t] \setminus \kk$. Let $g_f,h_f \in \kk[t]$ be as in Notation \ref{ntt:L(f,g)}, so that $f'g_f = h_f(f) \in \kk[f]$. If $\deg(f) > 1$, then $\deg(f'g_f) \geq 1$, so $h_f$ is non-constant. In this case, $L(f) \cong h_f\WW_1$ with $h_f \in \kk[t] \setminus \kk$, which contradicts the first paragraph. Therefore, $\deg(f) = 1$, so $L(f) = \WW_1$, and thus $L$ has finite codimension in $\WW_1$.
    
    By Proposition \ref{prop:Alexey}, there exist $g \in \kk[t]$ and $n \in \NN$ such that
    $$g^n\WW_1 \subseteq L \subseteq g\WW_1.$$
    By the first paragraph, we see that $g$ must be constant, so $L = \WW_1$.
\end{proof}

The computation of endomorphisms of $\WW_1$ now follows easily.

\begin{cor}\label{cor:Dixmier}
    The Lie algebra $\WW_1$ satisfies the Dixmier conjecture. In other words, $\End(\WW_1) \nonzero = \Aut(\WW_1)$.
\end{cor}
\begin{proof}
    Let $\varphi \in \End(\WW_1) \nonzero$. Since $\WW_1$ is a simple Lie algebra, $\varphi$ must be injective. By Proposition \ref{prop:simple subalgebra}, $\varphi$ must also be surjective.
\end{proof}

It would be interesting to see if a similar study of subalgebras of $\WW_2$ yields any information about Zhao's conjecture for two variables. In turn, this would give information about the two-dimensional Jacobian conjecture and the Dixmier conjecture for the first Weyl algebra.

\section[Wn]{Subalgebras of $\WW_n$}

In this section, we generalise Corollary \ref{cor:noetherian} to subalgebras of $\WW_n$ of Gelfand--Kirillov (GK) dimension $n$.

\begin{thm}\label{thm:subalgebras Wn}
    Let $n \geq 1$ and let $L$ be a subalgebra of $\WW_n$ of GK-dimension $n$. Then $U(L)$ is not noetherian.
\end{thm}

\begin{rem}
    We note that $\WW_n$ is itself a finitely generated Lie algebra of GK-dimension $n$, so the constraint that $L$ has GK-dimension $n$ simply says $L$ is in some sense a ``large'' subalgebra.
\end{rem}

The key to the proof of Theorem \ref{thm:subalgebras Wn} is to introduce subalgebras $L_{ij}$ of $L$ and consider the dimensions of successive quotients $L_{ij}/L_{i,j+1}$. Most of the results in this section are devoted to reducing to the case where almost all of these dimensions are finite. We will then conclude the proof by showing that this is impossible if $L$ has GK-dimension $n$.

\begin{ntt}
    Let $n \geq 1$ and let $L$ be a subalgebra of $\WW_n$. For $i \in \NN$ and $j = 0, \cdots, n - 1$, we let 
    $$L_{ij} = L \cap \left(\sum_{k \le j} x_1^{i+1}\kk[x_1,\cdots,x_n] \del_k  + \sum_{k>j} x_1^i\kk[x_1,\cdots,x_n] \del_k\right).$$
    We will take $L_{i,n} = L_{i+1,0}$. We also write $d_{ij}(L) = \dim_\kk(L_{ij}/L_{i,j+1})$.
\end{ntt}

Note that, for a subalgebra $L$ of $\WW_n$, we have $L_{ij} \supseteq L_{i,j+1}$ for $j = 0,\cdots, d - 2$, and $L_{i,n-1} \supseteq L_{i+1,0}$. We will ultimately have to apply a change of variables (i.e. an automorphism of $\WW_n$) to ensure that we can obtain the conditions needed for the proof of Theorem \ref{thm:subalgebras Wn} to work, but this will only be done at the end.

We will use the following lemma extensively to greatly simplify some proofs. In particular, in order to prove that the enveloping algebra of a Lie algebra $L$ is not noetherian, it suffices to prove this for a subalgebra of $L$.

\begin{lem}[{\cite[Lemma 1.7]{SierraWalton}}]\label{lem:subalgebra}
    Let $L_1$ be a Lie algebra and let $L_2$ be a Lie subalgebra of $L_1$. If $U(L_1)$ is noetherian then $U(L_2)$ is also noetherian.
\end{lem}

The next result shows that Lie algebras with infinite-dimensional abelianisations have non-noetherian enveloping algebras.

\begin{lem}\label{lem:derived subalgebra}
    Let $L$ be an infinite-dimensional Lie algebra. If $[L,L]$ has infinite codimension in $L$ then $U(L)$ is not noetherian.
\end{lem}
\begin{proof}
    The abelianisation $L^{\ab} = L/[L,L]$ is infinite-dimensional and abelian, so $U(L^{\ab})$ is not noetherian, since it is isomorphic to a polynomial ring in infinitely many variables. The natural map $U(L) \to U(L^{\ab})$ is surjective, so $U(L)$ is also not noetherian.
\end{proof}

In some sense, Lemma \ref{lem:derived subalgebra} says that Lie algebras that are ``almost abelian" have non-noetherian enveloping algebras. The Lie algebras $\WW_n$ have many such subalgebras for $n \geq 2$, but this is not the case for $\WW_1$. This gives some indication as to why the proof of Theorem \ref{thm:subalgebras Wn} is so different depending on whether $n = 1$ or $n \geq 2$.

Using Lemma \ref{lem:derived subalgebra}, we can now reduce to the case where $d_{ij}(L) < \infty$ for $i \geq 1$.

\begin{lem}\label{lem:d's are finite}
    Let $n \geq 1$ and let $L$ be a subalgebra of $\WW_n$ such that $d_{ij}(L) = \infty$ for some $i \geq 1$ and $j \in \{0, \cdots, n - 1\}$. Then $U(L)$ is not noetherian.
\end{lem}
\begin{proof}
    We claim that $[L_{ij},L_{ij}] \subseteq L_{i,j+1}$. Let
    $$u = \sum_{k \leq j} x_1^{i+1} u_k \del_k + \sum_{k > j} x_1^i u_k \del_k, \quad v = \sum_{\ell \leq j} x_1^{i+1} v_\ell \del_\ell + \sum_{\ell > j} x_1^i v_\ell \del_\ell \in L_{ij},$$
    where $u_k,v_\ell \in \kk[x_1,\cdots,x_n]$. Write $w_{k\ell} = [x_1^{i_k} u_k \del_k, x_1^{i_\ell} v_\ell \del_\ell]$, where $i_m = i + 1$ if $m \leq j$ or $i_m = i$ if $m > j$, so that $[u,v] = \sum_{k,\ell=1}^n w_{k\ell}$. We will prove the claim by showing that $w_{k\ell}$ is an element of
    $$(\WW_2)_{i,j+1} = \sum_{k \le j+1} x_1^{i+1}\kk[x_1,\cdots,x_n] \del_k  + \sum_{k>j+1} x_1^i\kk[x_1,\cdots,x_n] \del_k$$
    for all $k,\ell$. We will use the following identity:
    \beq\label{eq:expansion of bracket}
        [x_1^p f \del_k, x_1^q g \del_\ell] = x_1^{p+q-1}\Big((x_1 f \del_k(g) + \delta_{1,k} q fg)\del_\ell \\- (x_1 \del_\ell(f) g + \delta_{1,\ell} p f g)\del_k\Big),
    \eeq
    for $f,g \in \kk[x_1,\cdots,x_n]$ and $p,q \in \NN$. By \eqref{eq:expansion of bracket}, the only nontrivial case is when $i = 1, j = 0$, and at least one of $k$ and $\ell$ is 1. First, if exactly one of $k$ and $\ell$ is 1, say $k \geq 2$ and $\ell = 1$, then \eqref{eq:expansion of bracket} implies that
    $$w_{k,1} = [x_1 u_k \del_k,x_1 v_1 \del_1] = x_1^2 u_k \del_k(v_1) \del_1 - x_1(x_1 \del_1(u_k) v_1 + u_k v_1)\del_k \in (\WW_2)_{1,1}.$$
    If $k = \ell = 1$, then
    $$w_{1,1} = [x_1 u_1 \del_1,x_1 v_1 \del_1] = x_1^2(u_1 \del_1(v_1) - \del_1(u_1) v_1)\del_1 \in (\WW_2)_{1,1}.$$
    This proves the claim.
    
    It follows that the abelianisation $L_{ij}^{\ab} = L_{ij}/[L_{ij},L_{ij}]$ is infinite-dimensional, since $d_{ij}(L) = \infty$, so $U(L_{ij})$ is not noetherian by Lemma \ref{lem:derived subalgebra}. The result now follows by Lemma \ref{lem:subalgebra}.
\end{proof}

We now focus on reducing to $d_{0,0}(L)$ also being finite. This will be achieved by showing that $U(L)$ can only be noetherian if either $d_{0,0}(L) < \infty$ or $L_{2,0} = 0$. We then use the GK-dimension assumption to prove that $L_{2,0}$ is always nonzero.

\begin{prop}\label{prop:dichotomy}
    Let $n \geq 1$ and let $L$ be a subalgebra of $\WW_n$ such that $d_{0,0}(L) = \infty$ and $L_{k,0}/L_{k+1,0} \neq 0$ for some $k \geq 2$. Then $U(L)$ is not noetherian.
\end{prop}
\begin{proof}
    Assume, for a contradiction, that $U(L)$ is noetherian. The Lie bracket gives a bilinear map $\tau \colon L/L_{0,1} \times L_{k,0}/L_{k+1,0} \to L_{k-1,0}/L_{k,0}$. We emphasise that we are only regarding these as quotients of vector spaces, since $L_{0,1}$ is not a Lie ideal of $L$.

    Let $\mc{A} \coloneqq N_L(L_{k,0}) = \{\Phi \in L \mid [\Phi,L_{k,0}] \subseteq L_{k,0}\}$ be the normaliser of $L_{k,0}$ in $L$, and let $\overline{\mc{A}}$ be the image of $\mc{A}$ in $L/L_{0,1}$. We claim that $\overline{\mc{A}}$ is infinite-dimensional. Note that we may assume that $L_{k,0}/L_{k+1,0}$ and $L_{k-1,0}/L_{k,0}$ are finite-dimensional, by Lemma \ref{lem:d's are finite}. Let $w_1,\cdots,w_m$ be a basis for $L_{k,0}/L_{k+1,0}$. For $i = 1,\cdots, m$, define $\tau_i \colon L/L_{0,1} \to L_{k-1,0}/L_{k,0}$ by $\tau_i(v) = \tau(v,w_i)$, where $v \in L/L_{0,1}$. Since $L_{k-1,0}/L_{k,0}$ is finite-dimensional, it follows that $\ker(\tau_i)$ has finite codimension in $L/L_{0,1}$. Therefore, $\overline{\mc{A}} = \bigcap_{i=1}^m \ker(\tau_i)$ is infinite-dimensional, as claimed.

    Let $u = \sum_{i=1}^n u_i\del_i \in \mc{A}$ such that $u \notin L_{0,1}$, in other words, $u_1 \notin x_1 \kk[x_1,\cdots,x_n]$. Let $v = x_1^k\sum_{i=1}^n v_i\del_i \in L_{k,0} \setminus L_{k+1,0}$, which exists by assumption. By definition of $\mc{A}$, we have $[u,v] \in L_{k,0}$. Expanding $[u,v]$, we get
    \begin{align*}
        [u,v] &= \left[u_1\del_1,x_1^k \sum_{i=1}^n v_i\del_i\right] + \left[\sum_{i=2}^n u_i\del_i,x_1^k \sum_{i=1}^n v_i\del_i\right] \\
        &= kx_1^{k-1} u_1 \sum_{i=1}^n v_i\del_i + x_1^k\left(u_1\sum_{i=1}^n \del_1(v_i)\del_i - \sum_{i=1}^n v_i\del_i(u_1)\del_1 + \left[\sum_{i=2}^n u_i\del_i,\sum_{i=1}^n v_i\del_i\right]\right).
    \end{align*}
    Since $[u,v] \in L_{k,0} \subseteq x_1^k\WW_n$, we must have $kx_1^{k-1} u_1 \sum_{i=1}^n v_i\del_i \in x_1^k\WW_n$. We have $u_1 \notin x_1 \kk[x_1,\cdots,x_n]$ by assumption, and therefore $v_i \in x_1 \kk[x_1,\cdots,x_n]$ for all $i \in \{1,\cdots,n\}$. This is a contradiction, since $v \notin L_{k+1,0}$.
\end{proof}

\begin{cor}\label{cor:dichotomy}
    Let $n \geq 1$ and let $L$ be a subalgebra of $\WW_n$ such that $d_{0,0}(L) = \infty$ and $L_{2,0} \neq 0$. Then $U(L)$ is not noetherian.
\end{cor}
\begin{proof}
    Assume, for a contradiction, that $U(L)$ is noetherian. Proposition \ref{prop:dichotomy} implies that $L_{i,0}/L_{i+1,0} = 0$ for all $i \geq 2$, in other words, $L_{i,0} = L_{2,0}$ for all $i \geq 2$. By Krull's intersection theorem,
    $$L_{2,0} = \bigcap_{i \geq 2} L_{i,0} = 0,$$
    a contradiction.
\end{proof}

Note that, until now, all the results hold for any subalgebra $L \subseteq \WW_n$. The results that follow are specific to subalgebras of GK-dimension $n$. We first need some notation.

\begin{ntt}
    Let $V$ be a subspace of a Lie algebra $L$. We write $C_0(V) = V$ and
    $$C_N(V) = [C_{N-1}(V),V] = \spn\{[u,v] \mid u \in C_{N-1}(V), v \in V\},$$
    for $N \geq 1$.
\end{ntt}

We now prove that $L_{2,0} \neq 0$ if $\GKdim(L) = n$.

\begin{prop}\label{prop:L2,0 is infinite}
    Let $n \geq 1$ and let $L$ be a subalgebra of $\WW_n$ of GK-dimension $n$. Then $\dim_\kk(L_{2,0}) = \infty$.
\end{prop}
\begin{proof}
    Since $\GKdim(L) = n$, there is a finitely generated subalgebra $L'$ of $L$ of GK-dimension greater than $n - 1$. We let $u_i = \sum_{j=1}^n u_{ij}\del_j$ for $i = 1,\cdots,m$ be a set of generators of $L'$. Letting
    $$N = \max\{\deg(u_{ij}) \mid 1 \leq i \leq m, 1 \leq j \leq n\}$$
    and $V = \spn\{u_i \mid 1 \leq i \leq m\}$, we have
    $$C_k(V) \subseteq \sum_{j=1}^n\kk[x_1,\cdots,x_n]_{\leq Nk}\del_j,$$ where $\kk[x_1,\cdots,x_n]_{\leq \ell}$ denotes the space of polynomials of degree at most $\ell$. It is well-known that $\dim_\kk(\kk[x_1,\cdots,x_n]_{\leq \ell}) = \binom{n + \ell}{n}$.
    
    We claim that $\kk[x_1,\cdots,x_n]_{\leq Nk} \cap (x_1^2)$ has codimension $\binom{Nk + n - 1}{n - 1} + \binom{Nk + n - 2}{n - 1} = O(k^{n-1})$ in $\kk[x_1,\cdots,x_n]_{\leq Nk}$. To see this, note that
    $$\frac{\kk[x_1,\cdots,x_n]_{\leq Nk}}{\kk[x_1,\cdots,x_n]_{\leq Nk} \cap (x_1^2)} \cong \kk[x_2,\cdots,x_n]_{\leq Nk} \oplus x_1\kk[x_2,\cdots,x_n]_{\leq Nk-1}.$$
    Therefore,
    \begin{align*}
        \dim_\kk\left(\frac{\kk[x_1,\cdots,x_n]_{\leq Nk}}{\kk[x_1,\cdots,x_n]_{\leq Nk} \cap (x_1^2)}\right) &= \dim_\kk(\kk[x_2,\cdots,x_n]_{\leq Nk}) + \dim_\kk(\kk[x_2,\cdots,x_n]_{\leq Nk-1}) \\
        &= \binom{Nk + n - 1}{n - 1} + \binom{Nk + n - 2}{n - 1},
    \end{align*}
    as claimed. By the assumption on the GK-dimension of $L'$, there is some $\varepsilon > 0$ such that $\dim_\kk(C_k(V)) \geq k^{n-1+\varepsilon}$ for $k \gg 0$. Since $\dim_\kk(\frac{C_k(V)}{C_k(V) \cap x_1^2\WW_n}) = O(k^{n-1})$, we must have $\dim_\kk(L'_{2,0}) = \infty$. This concludes the proof.
\end{proof}

Combining Corollary \ref{cor:dichotomy} and Proposition \ref{prop:L2,0 is infinite}, we immediately deduce the following.

\begin{cor}\label{cor:d00}
    Let $n \geq 1$ and let $L$ be a subalgebra of $\WW_n$ of GK-dimension $n$. If $d_{0,0}(L) = \infty$ then $U(L)$ is not noetherian. \qed
\end{cor}

Thanks to the above results, we have reduced to the case where $d_{ij}(L)$ is finite for all $i \geq 1$ and $j \in \{0,\cdots,n - 1\}$, and $(i,j) = (0,0)$. Note that the same is true if we change variables; in other words, if $\varphi \in \Aut(\WW_n)$, we may assume that $d_{ij}(\varphi(L)) < \infty$ for all $i \geq 1$ and $j \in \{0,\cdots,n - 1\}$, and $(i,j) = (0,0)$. We now prove Theorem \ref{thm:subalgebras Wn} by showing that no such subalgebra $L \subseteq \WW_n$ with $\GKdim(L) = n$ exists.

\begin{proof}[Proof of Theorem \ref{thm:subalgebras Wn}]
    For $n = 1$, this is Corollary \ref{cor:noetherian}, so we assume that $n \geq 2$.
    
    Assume, for a contradiction, that $U(L)$ is noetherian. Note that Corollary \ref{cor:d00} applies for all choices of generators of $\kk[x_1,\cdots,x_n]$. In particular, letting $y_1 = x_1 - \lambda$ and $y_i = x_i$ for $i \geq 2$, the quotient
    $$L_\lambda \coloneqq \frac{L}{L \cap \left(y_1\kk[y_1,\cdots,y_n]\del_1 + \sum_{j = 2}^n \kk[y_1,\cdots,y_n]\del_j\right)},$$
    which generalises $L_{0,0}/L_{0,1} = L_{\lambda = 0}$, is finite-dimensional for all $\lambda \in \kk$. If $\kk$ is uncountable, then it is immediate that there exist $N \in \NN$ and an infinite subset $U \subseteq \kk$ such that $\dim_\kk(L_\lambda) \leq N$ for all $\lambda \in U$. However, if $\kk$ is countable we require a different argument.

    Let $\overline{L} = L \otimes_{\kk} \kk(t)$, viewed as a Lie algebra over $\kk(t)$. This is a Lie subalgebra of
    $$\WW_n^{\kk(t)} = \Der_{\kk(t)}(\kk(t)[x_1,\cdots,x_n]),$$
    where $\Der_{\kk(t)}$ denotes $\kk(t)$-linear derivations. Since $U(L)$ is assumed to be noetherian, it follows that $\UU(\overline{L}) \cong U(L) \otimes_\kk \kk(t)$ is noetherian, by Hilbert's basis theorem, where $\UU(\overline{L})$ denotes the universal enveloping algebra of $\overline{L}$ as a Lie algebra over $\kk(t)$.

    The results above only assumed that $\kk$ is a field of characteristic zero, so they all apply to $\overline{L}$ and $\UU(\overline{L})$. Therefore, letting $y_1 = x_1 - \chi(t)$ and $y_i = x_i$ for $i \geq 2$, the quotient
    $$\overline{L}_{\chi(t)} \coloneqq \frac{\overline{L}}{\overline{L} \cap \left(y_1\kk(t)[y_1,\cdots,y_n]\del_1 + \sum_{j = 2}^n \kk(t)[y_1,\cdots,y_n]\del_j\right)}$$
    is a finite-dimensional vector space over $\kk(t)$ for all $\chi(t) \in \kk(t)$, similarly to above. In particular, $\overline{L}_t$ is finite-dimensional. Let $N = \dim_{\kk(t)}(\overline{L}_t)$.
    
    Let
    $$X = \overline{L} \cap \left((x_1 - t)\kk(t)[x_1,\cdots,x_n]\del_1 + \sum_{j = 2}^n \kk(t)[x_1,\cdots,x_n]\del_j\right),$$
    so that $\overline{L}_t = \overline{L}/X$. Note that $X$ must be finitely generated: if not, then there exists an infinite strictly ascending chain of (finitely generated) subalgebras of $X$, which would imply that $\UU(X)$ is not noetherian by \cite[Proposition 11.1.2]{AmayoStewart}, a contradiction to Lemma \ref{lem:subalgebra}. Let
    $$a_i = (x_1 - t)a_{i,1}\del_1 + \sum_{j=2}^n a_{ij}\del_j$$
    with $i = 1,\cdots,\ell$ be a set of generators of $X$, where $a_{ij} \in \kk(t)[x_1,\cdots,x_n]$.

    Pick elements $b_i = c_i \otimes 1 \in \overline{L}$ for $i = 1,\cdots,N$, where $c_i \in L$, such that
    \begin{equation}\label{eq:codimension of X}
        \overline{L} = \kk(t) b_1 \oplus \cdots \oplus \kk(t) b_N \oplus X.
    \end{equation}
    Let $\overline{w}_1 = w_1 \otimes 1,\cdots,\overline{w}_r = w_r \otimes 1$ be generators of $\overline{L}$, where $w_i \in L$. The equality \eqref{eq:codimension of X} can be witnessed as follows:
    \begin{enumerate}
        \item For each $i = 1,\cdots,r$, we have an expression
        \begin{equation}\label{eq:generators are contained}
            \overline{w}_i = \sum_{j=1}^N \alpha_{ij}(t) b_j + (x_1 - t)f_{i,1} \del_1 + \sum_{j=2}^n f_{ij} \del_j,
        \end{equation}
        for some $\alpha_{ij}(t) \in \kk(t)$ and $f_{ij} \in \kk(t)[x_1,\cdots,x_n]$.
        \item For each $i,j$,
        \begin{equation}\label{eq:closed part 1}
            [b_i,b_j] = \sum_{s=1}^N \beta_{ijs}(t) b_s + (x_1 - t)g_{i,j,1} \del_1 + \sum_{s=2}^n g_{ijs}\del_s,
        \end{equation}
        for some $\beta_{ijs}(t) \in \kk(t)$ and $g_{ijs} \in \kk(t)[x_1,\cdots,x_n]$.
        \item For each $i,j$,
        \begin{equation}\label{eq:closed part 2}
            [a_i,b_j] = \sum_{s=1}^N \gamma_{ijs}(t) b_s + (x_1 - t) h_{i,j,1} \del_1 + \sum_{s=2}^n h_{ijs} \del_s,
        \end{equation}
        for some $\gamma_{ijs} \in \kk(t)$ and $h_{ijs} \in \kk(t)[x_1,\cdots,x_n]$.
    \end{enumerate}
    Indeed, \eqref{eq:generators are contained} is saying that $\kk(t) b_1 \oplus \cdots \oplus \kk(t) b_N \oplus X$ contains a set of generators of $\overline{L}$, while \eqref{eq:closed part 1} and \eqref{eq:closed part 2} imply that $\kk(t) b_1 \oplus \cdots \oplus \kk(t) b_N \oplus X$ is closed under the Lie bracket of $\WW_n^{\kk(t)}$, giving the desired equality. Now, there are only finitely many polynomials $\alpha_{ij}(t), \beta_{ijs}(t), \gamma_{ijs}(t), a_{ij}, f_{ij}, g_{ijs},h_{ijs}$, so there are only finitely many poles of the coefficients of these polynomials.
    
    In particular, \eqref{eq:generators are contained} implies that for all but finitely many specializations $t = \lambda$ with $\lambda \in \kk$, we have a relation:
    $$w_i = \sum_{j=1}^N \alpha_{ij}(\lambda)c_j + (x_1 - \lambda) \widetilde{f}_{i,1} \del_1 + \sum_{j=2}^n \widetilde{f}_{ij} \del_j,$$
    where $\widetilde{f}_{ij}$ denotes $f_{ij}$ after we substitute $t = \lambda$ into all the coefficients. This implies that a set of generators of $L$ is contained in $\kk c_1 + \cdots + \kk c_N + \widetilde{X}_\lambda$, where
    $$\widetilde{X}_\lambda = L \cap \left((x_1 - \lambda)\kk[x_1,\cdots,x_n]\del_1 + \sum_{j = 2}^n \kk[x_1,\cdots,x_n]\del_j\right).$$
    Similarly, we can get relations corresponding to \eqref{eq:closed part 1} and \eqref{eq:closed part 2}, which tells us that $\kk c_1 + \cdots + \kk c_N + \widetilde{X}_\lambda$ is closed under the Lie bracket of $\WW_n$. Therefore,
    $$L = \kk c_1 + \cdots + \kk c_N + \widetilde{X}_\lambda,$$
    which means that $\dim_\kk(L_{\lambda}) \leq N$ for all but finitely many $\lambda \in \kk$.

    In other words, even when $\kk$ is countable, there exist $N \in \NN$ and an infinite subset $U \subseteq \kk$ such that $\dim_\kk(L_\lambda) \leq N$ for all $\lambda \in U$.

    We claim that $\dim_\kk(F_1) \leq N$, where
    $$F_1 = \spn\{f \in \kk[x_1,\cdots,x_n] \mid f\del_1 + \sum_{j=2}^n g_j\del_j \in L \text{ for some } g_j \in \kk[x_1,\cdots,x_n]\}.$$
    Assume, for a contradiction, that there exist $N + 1$ elements $u_i = \sum_{j=1}^n u_{ij}\del_j \in L$ such that $u_{i,1}$ are linearly independent for $i = 1,\cdots,N + 1$. Write $v_i = u_{i,1}$ for $i = 1,\cdots,N + 1$. Since $\dim_\kk(L_\lambda) \leq N$ for $\lambda \in U$, we must have that $v_1,\cdots,v_{N+1}$ are linearly dependent modulo $x_1 - \lambda$ for all $\lambda \in U$. Write
    $$v_i = \sum_{\mathbf{k} \in \NN^{n-1}} P_{\mathbf{k}}^{(i)}(x_1)x_2^{k_2} \cdots x_n^{k_n},$$
    where $\mathbf{k} = (k_2,\cdots,k_n)$ and $P_{\mathbf{k}}^{(i)}(x_1) \in \kk[x_1]$. Let
    $$I = \{\mathbf{k} \in \NN^{n-1} \mid P_{\mathbf{k}}^{(i)}(x_1) \neq 0 \text{ for some } i \in \{1,\cdots,N+1\}\}.$$
    For $i \in \{1,\cdots,N + 1\}$ and $\mathbf{k} \in I$, let $c_{i,\mathbf{k}} = P_{\mathbf{k}}^{(i)}(x_1)$. Let $C = (c_{i,\mathbf{k}}) \in M_{(N+1) \times |I|}(\kk[x_1])$ be a matrix with $N + 1$ rows, and columns labelled by elements of $I$. Since $v_1,\cdots,v_{N+1}$ are linearly independent, we have $|I| \ge N + 1$ and there is some $(N + 1) \times (N + 1)$ nonzero minor $P \in \kk[x_1]$ of $C$. On the other hand, $P \in (x_1 - \lambda)\kk[x_1]$ for all $\lambda \in U$, since $v_1(\lambda,x_2,\cdots,x_n),\cdots,v_{N+1}(\lambda,x_2,\cdots,x_n)$ are linearly dependent. But then
    $$P \in \bigcap_{\lambda \in U} (x_1-\lambda)\kk[x_1] = 0,$$
    a contradiction.
    
    Let
    $$F_i = \spn\{f \in \kk[x_1,\cdots,x_n] \mid f\del_i + \sum_{\substack{j=1 \\ j \neq i}}^n g_j\del_j \in L \text{ for some } g_j \in \kk[x_1,\cdots,x_n]\}.$$
    We have shown that $F_1$ is finite-dimensional. Since we can do this for all variables, we see that $F_i$ is finite-dimensional for all $i \in \{1,\cdots,n\}$, and therefore $L$ is finite-dimensional, a contradiction.
\end{proof}

\end{document}